\documentclass[12pt]{article}

\usepackage[english]{babel}
\usepackage[utf8]{inputenc}
\usepackage{amsmath}
\usepackage{graphicx}
\usepackage{amssymb}
\usepackage{amsthm}
\usepackage{tikz-cd}
\usepackage{mathrsfs}
\usepackage[colorinlistoftodos]{todonotes}
\usepackage[inline]{enumitem}
\usepackage{yfonts}
\usepackage{xcolor}
\usepackage{mathtools}
\usepackage{hyperref}
\usepackage{amsthm}

\title{On the Asymptotic Number of Generators of High Rank Arithmetic Lattices }

\author{Alexander Lubotzky, Raz Slutsky \\ 
~\\ \small{
\textit{Dedicated to Gopal Prasad on his 75th birthday }
}}
\date{}

\makeatletter
\def\thmheadbrackets#1#2#3{%
  \thmname{#1}\thmnumber{\@ifnotempty{#1}{ }\@upn{#2}}%
  \thmnote{ {\the\thm@notefont[#3]}}}
\makeatother

\newtheoremstyle{brakets}
  {}
  {}
  {\itshape}
  {}
  {\bfseries}
  {.}
  { }
  {\thmheadbrackets{#1}{#2}{#3}}

\theoremstyle{brakets}

\newtheorem{thm}{Theorem}[section]
\newtheorem{lem}[thm]{Lemma}
\newtheorem{prop}[thm]{Proposition}

\newtheorem{defn}[thm]{Definition}

\newtheorem{cor}[thm]{Corollary}

\newcommand{\ie}{\emph{i.e.}}

\newcommand{\R}{\mathbb{R}}

\newcommand{\Z}{\mathbb{Z}}
\newcommand{\N}{\mathbb{N}}
\newcommand{\Q}{\mathbb{Q}}

\newcommand{\Addresses}{{
  \bigskip
  \footnotesize

  Alexander Lubotzky , \textsc{Institute of Mathematics, Hebrew University, Jerusalem 91904, Israel. }\par\nopagebreak
  \textit{E-mail address},: \texttt{alex.lubotzky@mail.huji.ac.il}

  \medskip

  Raz Slutsky, \textsc{Department of Mathematics. Weizmann Institute of Science. Rehovot 76100, Israel.}\par\nopagebreak
  \textit{E-mail address}:  \texttt{razslo@gmail.com}

}}

\begin{document}

\maketitle

\begin{abstract}
Abert, Gelander and Nikolov \cite{agn17} conjectured that the number of generators $d(\Gamma)$ of a lattice $\Gamma$ in a high rank simple Lie group $H$ grows sub-linearly with $v = \mu(H / \Gamma)$, the co-volume of $\Gamma$ in $H$. We prove this for non-uniform lattices in a very strong form, showing that for $2-$generic such $H$'s, $d(\Gamma) = O_H(\log v / \log \log v)$, which is essentially optimal. While we can not prove a new upper bound for uniform lattices, we will show that for such lattices one can not expect to achieve a better bound than $d(\Gamma) = O(\log v)$.
\end{abstract}

\section{Introduction}

Let $H$ be a connected, non-compact, simple real Lie group, with a fixed Haar measure $\mu$. A discrete subgroup $\Gamma$ of $H$ is called a \emph{lattice} if $\mu(H/\Gamma) < \infty$. It is called uniform (or co-compact) if $H/ \Gamma$ is compact, and non-uniform otherwise. 
In \cite{Ge11}, Gelander showed that there exists a constant $C_1 = C_1(H)$ such that $d(\Gamma) \leq C_1 \: \mu(H/ \Gamma)$ where $d(\Gamma)$ is the minimal number of generators of such $\Gamma$. Recently, this was shown for other types of Lie groups \cite{gs20}. \\

In \cite[Conjecture 3]{agn17}, Abert, Gelander and Nikolov conjectured that if $H$ is of high rank, \ie, $\text{rank}_\R(H) \geq 2$, then $d(\Gamma)$ grows sub-linearly with $\mu(H/\Gamma)$.\\

The main goal of the present paper is to prove a strong form (essentially optimal) of this conjecture for the \textbf{non-uniform} lattices of $H$. We will make some remarks on the uniform case, but at this stage, we are unable to prove the conjecture for this case. We do, however, give a lower bound on $d(\Gamma)$ for uniform lattices, demonstrating a distinction between the growth rates of the two classes.\\

Following \cite{bl19}, we say that $H$ is 2-generic if the centre of the simply connected cover $\tilde{H}$ of the split form of $H$ is a 2-group, and $\tilde{H}$ has no outer automorphisms of order three. This is the case for "most" $H$'s. In fact it holds for all $H$ unless it is of type $E_6, D_4$ or $A_n$ for $n \neq 2^m -1$. \\

For convenience, and without loss of generality, we assume from now on that $\mu$, the Haar measure on $H$, is normalized so that $\mu(H/ \Gamma) > 1$ for every lattice $\Gamma$ in $H$. This is possible since there is a lower bound on the co-volume of lattices due to Kazhdan and Margulis \cite{KM}. In addition, throughout this paper all logarithms are in base 2. \\

Our main result is the following.

\begin{thm}\label{main theorem}
Let $H$ be a simple Lie group with $\R$-$\text{rank}(H) \geq 2$. Then there exists a constant $C_2 = C_2(H)$ such that 
\begin{itemize}
\item[(a)] For every non-uniform lattice $\Gamma$ of $H$ we have $$ d(\Gamma) \leq C_2 \log (\mu(H/\Gamma))$$
\item[(b)] If $H$ is 2-generic, then for every non-uniform lattice $\Gamma$ of $H$ we have $$d(\Gamma) \leq C_2 \dfrac{\log (\mu (H/ \Gamma))}{\log \log (\mu(H/ \Gamma))} $$
\end{itemize}
\end{thm}

We remark that we believe that the sharper bound of $(b)$, which is best possible (see Section \ref{lower bound section}), holds for all $H$, but it depends (and actually equivalent to) some delicate number-theoretic conjectures. More precisely, Gauss' celebrated theorem \cite[Theorem 8, p. 247]{bs66} gives a very precise description of the order of the 2-torsion of the class group $Cl(k)$ of quadratic number fields $k$. From this theorem one deduces that $d_2(Cl(k)) = O (\frac{\log D_k}{\log \log D_k})$ where $d_2(Cl(k))$ is the number of generators of the $2-$Sylow subgroup of $Cl(k)$ and $D_k$ is the absolute value of the discriminant $\Delta_k$ of $k$. Now, if we would know such bounds for the $p-$Sylow subgroups for the odd primes, then the estimate of $(b)$ would follow for all $H$. In fact, this is essentially equivalent (See \cite[Sections 3 and 7]{bl19}).\\

However, despite much effort over the years, the current knowledge is quite far from having such bounds (see \cite{ptbw20}). \\

The intimate connection between the group-theoretic/geometric statement of Theorem 1.1 and the delicate number theory is not as surprising as it seems at first sight. By the Margulis Arithmeticity Theorem \cite{margulis91}, every $\Gamma$ in a high rank group $H$ is an arithmetic lattice, and unlike the methods of \cite{Ge11} and \cite{agn17}, our method will use this fact extensively.\\

Our second result is based on the existence of infinite class field towers of totally real fields, due to Golod and Shafarevich \cite{gs64}, and follows the lines of \cite{bl12}:

\begin{thm}
\label{lower bound}
Let $H$ be a simple Lie group of high rank. Then there exists a constant $C_3 = C_3(H)$ and a sequence of uniform lattices $\Gamma_i \leq H$ with $\mu(H/\Gamma_i) \rightarrow \infty$ such that
$$ d(\Gamma_i) \geq C_3 \log \mu(H/\Gamma_i)  $$

\end{thm}

This lower bound shows that the growth rate of $d(\Gamma)$ for uniform lattices is strictly larger than that of non-uniform ones, thus establishing a further distinction between the two types of lattices.\\

As in \cite{bl12} and \cite{bl19}, the distinction stems from the fact that non-uniform lattices in $H$ are defined over number fields of bounded degree over $\Q$, while the degrees of the number fields defining uniform lattices are unbounded. \\

Let us now describe the main line of the proof of Theorem \ref{main theorem}.\\

Venkataramana \cite{ve87, ve94} has developed a method to show that various subgroups of arithmetic groups are of finite index. This method uses unipotent elements and hence is valid only for non-uniform arithmetic lattices. This accumulates to the following result of Sharma and Venkataramana \cite{sv05} which is a first main ingredient of our proof:

\begin{thm}[SV05, Theorem 1] \label{Sharma Venkataramana} Every high rank non-uniform arithmetic group $\Gamma$ has a subgroup of finite index which is generated by at most three elements.

\end{thm}

It follows that if $\hat{\Gamma}$ is the pro-finite completion of $\Gamma$ and $d(\hat{\Gamma})$ denotes the minimal number of topological generators of $\hat{\Gamma}$, then 
$$
d(\hat{\Gamma}) \leq d(\Gamma) \leq d(\hat{\Gamma}) + 3
$$

Thus, it suffices to prove Theorem \ref{main theorem} for $d(\hat{\Gamma})$ rather then $d(\Gamma)$. Now, by a standard inverse limit argument, $d(\hat{\Gamma})$ is the supremum over $d(S)$ where $S$ runs over finite quotients of $\Gamma$. Moreover, by a well-known result of Raghunathan \cite{raghunathan76}, non-uniform $\Gamma$'s satisfy the congruence subgroup property (denoted CSP from this point onwards), so we have to deal only with quotients by congruence subgroups. The proof for those will use some methods and techniques from \cite{bl12, bl19} which in turn use crucially the seminal work of Prasad \cite{prasad}. These are valid in almost the same way for uniform and non-uniform lattices. Thus, the obstacle that prevents us from proving the conjecture in its full generality is the use of Theorem \ref{Sharma Venkataramana} and the fact that the CSP is known only for some uniform lattices.\\

For simplicity of the introduction, we treated in this introduction the case where $H$ is simple, but similar results and methods apply to the case where $H$ is semi-simple and $\Gamma$ runs over all irreducible lattices, see Section \ref{lower bound section}. \\

\emph{This paper is dedicated to Gopal Prasad with admiration and affection. 
Prasad has made fundamental contributions to the arithmetic theory of 
algebraic groups. In particular, this paper is based on his seminal work on the co-volume of arithmetic lattices.} \\

\noindent\textbf{Acknowledgments.} 
We would like to thank Andrei Rapinchuk and Igor Rapinchuk for helpful discussions and references. \\
The first author is indebted for support from the NSF (Grant No. DMS-1700165) and the European Research Council (ERC) under the European Union's Horizon 2020 research and innovation program (Grant No. 692854). This material is based upon work supported by a grant from the Institute for Advanced Study. \\
The second author is indebted for support from the Israel Science Foundation grant ISF 2919/19, and would like to thank Tsachik Gelander for his continued support and guidance.
\vspace{\baselineskip}

\section{Principal Arithmetic Groups and Congruence Subgroups}
\label{Section Principal}
Let $H$ be a high rank simple connected linear Lie group,
 and let $G$ be a semi-simple, simply connected, connected algebraic group defined over a number field $k$,
 with an epimorphism 
 $\phi: G( k \otimes_\Q \R) \rightarrow H$ whose kernel is compact. Then $\phi(G(\mathcal{O}))$ and subgroups of $H$ which are commensurable to it are called \emph{arithmetic}. All irreducible lattices in higher rank arise as such \cite{margulis91}, and one of the key facts for our purposes is that for non-uniform lattices, the degree of $k$ is bounded, where the bound depends only on $H$. Indeed, this follows from the well known result that $H/ \Gamma$
is non-compact if and only if $\Gamma$ contains non-trivial unipotent elements (see \cite[Chapter 5.3]{dave}). Hence, $G( k \otimes_\Q \R)$ has no compact factors, which implies that the number of Archimedean completions of $k$ is bounded by the number of simple factors of H and so $[k : \Q]$ is bounded, and when $H$ is simple, $[k:\mathbb{Q}] \leq 2$ .\\
 
Let now $\Gamma_0$ be a maximal lattice in $H$.
It is known (see \cite[Prop. 1.4]{bp89}) that in this setting, $\Gamma_0 = N(\Lambda')$ where $\Lambda'$ is a nice arithmetic group, namely the \textit{principal arithmetic subgroup} associated to a coherent family $(P_v)_{v \in V_f}$ of Parahoric subgroups in $G(k_v)$. More precisely, 

$$\Lambda = \Lambda(P) = G(k) \cap \prod_{v \in V_f} P_v $$
where $G$ is a $k-$form of $H$, $V_f$ is the set of finite places of $k$, $P_v \subset G(k_v)$, and $\Lambda'$ is the image of $\Lambda$ in $H$, namely $\phi(\Lambda) = \Lambda' \subset H$. \\

Now, for every $v$ there exists a smooth affine group scheme $G_v$ defined over $\mathcal{O}_v$ such that $G_v(\mathcal{O}_v) = P_v$ and $G_v(k_v)$ is $k_v$-isomorphic to $G(k_v)$. This induces a congruence subgroup structure on $P_v$ defined as:

$$P_v(r) = \ker (G_v(\mathcal{O}_v)  \rightarrow G_v(\mathcal{O}_v / \pi_v^r \mathcal{O}_v)) $$
where \(\pi_{v}\) is a uniformizer of \(\mathcal{O}_{v}\). These congruence subgroups induce a congruence
structure on \(\Lambda \), namely \( \Lambda\left(\pi_{v}^{r}\right)=\mathrm{P}_{v}(r) \cap \Lambda .\) More generally, for every ideal \(I\) of \(\mathcal{O}\) look at its
closure \(\bar{I}\) in \(\hat{\mathcal{O}}=\prod_{v} \mathcal{O}_{v}\). Then \(\bar{I}\) is equal to \(\prod_{i=1}^{l} \pi_{v_{i}}^{e_{i}} \hat{\mathcal{O}}\) for some \(Y=\left\{v_{1}, \ldots, v_{l}\right\} \subset V_{f}\)
and \(e_{1}, \ldots, e_{l} \in \mathbb{N}\). We then define the \(I\)-congruence subgroup of \(\Lambda,\)
$$
\Lambda(I)=\Lambda \cap\left(\prod_{i=1}^{l} \mathrm{P}_{v_{i}}\left(e_{i}\right) \cdot \prod_{v \notin Y} \mathrm{P}_{v}\right)
$$

In particular, for every $m \in \mathbb{N}$, the $m-$congruence subgroup is defined as $\Lambda(m) := \Lambda(m \mathcal{O})$, and any subgroup of $\Lambda$ which contains $\Lambda(I)$ for some $0 \neq I \triangleleft \mathcal{O} $ is called a congruence subgroup.\\

In the next few sections we use  the congruence structure of such lattices to prove the main theorem. We will study $\Lambda$ and $\Lambda'$ interchangeably since an upper bound on $d(\Gamma)$ for $\Gamma$ in $\Lambda$ gives a similar bound on $d(\Gamma')$ for $\Gamma' = \phi(\Gamma)$ in $\Lambda'$.

\subsection{Reduction to Subgroups of Principal Arithmetic Lattices}
First, we use the fact that the index of these principal arithmetic subgroups $\Lambda' \leq \Gamma_0$ is bounded by a function of the co-volume of $\Gamma_0$, and so we can work inside the principal arithmetic lattice while paying a negligible price. More precisely, we have by \cite[Prop. 4.1]{bl19}:

\begin{prop}
\label{Index of Principal}
There exists a constant $C_4 = C_4(H)$ such that for every $Q = \Gamma_0 / \Lambda'$ with $\mu(H/\Gamma_0) = v$ as above,

\begin{itemize}
\item[(i)] $|Q| \leq v^{C_4} $
\item[(ii)] if $\Gamma_0$ is non-uniform and $H$ is 2-generic, then $|Q| \leq C_4^{\log v / \log \log v}$.
\end{itemize}

\end{prop}

This enables us to reduce our main theorem to the following:

\begin{thm}
\label{first reduction}
There exists a constant $C_5 = C_5(H)$ such that  if $\Lambda'$ is a non-uniform principal arithmetic group as above, and $\Gamma \leq \Lambda'$ a finite index subgroup of co-volume $v$, then $d(\Gamma) \leq C_5 \log(v)$. Furthermore, if $H$ is 2-generic, we have $d(\Gamma) \leq C_5 \log (v) / \log \log (v)$. 
\end{thm}

Let us first show how Theorem \ref{first reduction} and Proposition \ref{Index of Principal} imply the main Theorem \ref{main theorem}.\\

\begin{prop}
Theorem \ref{first reduction} and Proposition \ref{Index of Principal} imply Theorem \ref{main theorem}.
\end{prop}
\begin{proof}

Assume that $H$ is $2-$generic. Let $\Gamma$ be a lattice in $H$ of co-volume $v$ and $\Gamma_0$ a maximal lattice containing it. 

Let $\Gamma_0 = N(\Lambda')$ as above and $\Gamma' = \Gamma \cap \Lambda'$. By Proposition \ref{Index of Principal}, $$[\Gamma : \Gamma'] \leq C_4^{\log v / \log \log v} \leq v^{\log C_4} $$

Now, since $\mu(H / \Gamma) \leq v$ we have $$\mu(H / \Gamma') \leq v^{1+\log C_4} $$ 

We can now use Theorem \ref{first reduction}, to deduce:
\begin{align*}
d(\Gamma') \leq & C_5((\log v^{1+\log C_4})/ \log \log v^{1+ \log C_4})  \\ \leq & C_5(1+\log C_4) \dfrac{\log v}{\log \log v} = C_6 \dfrac{\log v}{\log \log v}
\end{align*} 

Now, $\Gamma' = \Gamma \cap \Lambda' \trianglelefteq \Gamma$ and $\Gamma / \Gamma' \cong \Lambda' \Gamma / \Lambda'$. Using again Prop. \ref{Index of Principal} we deduce that 
$$d(\Gamma / \Gamma') \leq (\log C_4) \dfrac{\log v}{\log \log v} $$

Hence $d(\Gamma) \leq d(\Gamma') + d(\Gamma / \Gamma') \leq  (C_6+\log C_4) \dfrac{\log v}{\log \log v}$.\\

The case when $H$ is not $2$-generic is similar and even slightly simpler.

%
%
%
%
%
%
\end{proof}


\section{Proof of the Upper Bound}
\label{section of proof}
\subsection{Reduction to Congruence Subgroups}
\label{section of congruence}
We wish to use the congruence subgroup property and Theorem \ref{Sharma Venkataramana} to further reduce the main theorem to a question about finite quotients of congruence subgroups. Notice that the arguments so far did not use the fact that the lattices are non-uniform. However, from now on, we are going to use Theorem \ref{Sharma Venkataramana}, which is known only for non-uniform lattices, and the positive answer to the congruence subgroup problem for such lattices due to \cite{raghunathan76}.\\

First of all, we are going to pass to the pro-finite completion of $\Gamma$ and $\Lambda$. This is due to the result of Sharma-Venkataramana in Theorem \ref{Sharma Venkataramana}, which says that every non-uniform lattice has a finite index subgroup which is generated by at most 3 elements. It follows that for such lattices $d(\Gamma) \leq 3 + d(\widehat{\Gamma})$. In other words, it is enough to bound the number of generators of all finite quotients of $\Gamma$.
\\

The second important assumption is the Congruence Subgroup Property.\\

Recall that $\Lambda = G(k) \cap \prod_{v \in V_f} P_v$. By the CSP, its pro-finite completion $\widehat{\Lambda}$ is essentially equal to its congruence completion. To be more precise, the congruence kernel, $C = \ker (\widehat{\Lambda} \rightarrow \prod_{v \in V_f} P_v)$ is finite, hence, $d(\widehat{\Lambda}) \leq d(C) + d(\overline{\Lambda})$ where $\overline{\Lambda}$ is the congruence completion of $\Lambda$. \\

Now, the result of \cite{pr96} shows that the Margulis-Platonov conjecture and the CSP conjecture imply that $C$ is cyclic. Thus, $d(\widehat{\Lambda}) \leq 1 + d(\overline{\Lambda})$. In the non-uniform case, both the Margulis-Platonov and the CSP conjecture are known, see \cite{raghunathan76,pr10}. In addition, the same inequality holds for every finite index subgroup $\Gamma$ of $\Lambda$. Hence altogether $d(\Gamma) \leq d(\widehat{\Gamma}) + 3 \leq d(\overline{\Gamma}) +4$, where $\overline{\Gamma}$ is the congruence completion of $\Gamma$, which is in fact equal to the closure of $\Gamma$ in $\prod_{v \in V_f} P_v$, by strong approximation \cite[Thm 7.12]{pr93}.
\\

Working now with the congruence completion and congruence subgroups, we use a quantitative version of the "level versus index" lemma, in order to pass to principal congruence subgroups.\\

Recall the classical lemma asserting that in $SL_2(\Z)$, every congruence subgroup of index $n$ contains $\Delta(m) = \ker (SL_2(\Z) \rightarrow SL_2(\Z / m\Z))$ for some $m \leq n$. This lemma was later generalized in a quantitative manner in \cite[Lemma 4.3, Remark 4.4]{bl12}: 

\begin{lem}
\label{Level vs Index}
Let $\Lambda$ be a principal arithmetic group in $H$ with $\mu(H/ \Lambda) \leq v$, then if $\Lambda_1$ is a congruence subgroup of $\Lambda$ of index $n$, then $\Lambda(m\mathcal{O}) \subset \Lambda_1$ where $m \leq v^{C_7}n$, $C_7$ depends only on $H$, and $\Lambda(m\mathcal{O})$ is the principal congruence subgroup of level $m$. \\

In addition, if we restrict to non-uniform lattices, we have that $$[\Lambda:\Lambda(m\mathcal{O})] \leq (vn)^{C_8}$$
\end{lem}

Let now $\Gamma \leq \Lambda$ be a congruence subgroup with $\mu(G/ \Gamma) \leq v$. By our assumption on $\mu$ we have that $\mu(G / \Lambda) \geq 1$, thus the index of $\Gamma$ in $\Lambda$ is also bounded by $v$, and using \ref{Level vs Index} we have 
$$d(\Gamma) \leq d(\Lambda(m \mathcal{O})) + d(\Gamma / \Lambda(m \mathcal{O}))$$
where $m \leq v^{C_7+1}$. We shall now analyse the number of generators of these two factors.\\

\subsection{Rank of Principal Congruence Subgroups and of Finite Congruence Quotients}
So far, we have reduced the problem to bounding:\\

\begin{enumerate*}[label=\textnormal{(\roman*)}]
\item  $d(\Lambda(m \mathcal{O}))$, and \label{(i)}

\item  $d(\Gamma / \Lambda(m \mathcal{O}))$ \label{(ii)} 
\end{enumerate*} \\

In order to do so we will need the following definition and proposition:

\begin{defn}
Let $G$ be a pro-finite group, then
the \emph{Pr\"{u}fer rank}  or \emph{subgroup rank} of $G$ is defined as:
$$\text{rank}(G) :=  \sup  \{d(H) \; | \; H \leq G \} $$
where $H$ runs over the closed subgroups of $G$ and $d(H)$ denotes the minimal number of topological generators of $H$.

\end{defn}

It is easy to see that if $H \leq G$ and $K$ is a quotient of $H$ then $\text{rank}(K) \leq \text{rank}(G)$. Also, if $1 \rightarrow G_1 \rightarrow G_2 \rightarrow G_3 \rightarrow 1$ is an exact sequence of pro-finite groups then $\text{rank}(G_2) \leq \text{rank}(G_1) + \text{rank} (G_3)$.\\

The following proposition will be used several times later on:

\begin{prop}
\label{rank fo SL}
Let $d, s \in \N$, then $\exists \; C_9 = c(d,s) \leq 3d^2 s^2 $ such that if $k / \Q$ is a number field of degree $d$ and $\mathcal{O}$ its ring of integers, then
\begin{itemize}
\item[(a)] For every finite place $v \in V_f$ of $k$, $$\text{rank} \:(SL_s(\mathcal{O}_v)) \leq  C_9$$

\item[(b)] For every rational prime $p$,
$$\text{rank} \: (\prod_{v|p} SL_s(\mathcal{O}_v)) \leq d \cdot C_9 $$
\end{itemize}

\end{prop}

\begin{proof}
Denote the first congruence subgroup by $\Gamma_v(1) := \ker ( SL_s(\mathcal{O}_v) \rightarrow SL_s(\mathbb{F}_{q_v}))$ where $\mathbb{F}_{q_v}$ is the residue field of $\mathcal{O}_v$.  Since $\mathcal{O}_v$ is the integral closure of $\mathbb{Z}_p$ in $k_v$, we can embed $SL_s(\mathcal{O}_v)$ in $SL_{sd}(\mathbb{Z}_p)$. The congruence subgroup $\Gamma_v(1)$  is thus a subgroup of $\ker (SL_{sd}(\mathbb{Z}_p) \rightarrow SL_{sd}(\mathbb{F}_p))$ which is powerful pro-p of rank $\leq s^2d^2$ \cite[Theorem 5.2, and Theorem 3.8]{ddms}, thus $d(\Gamma_v(1)) \leq s^2d^2$. 
Now, since $\emph{rank} \: (SL_s(\mathcal{O}_v)) \leq \emph{rank}\: (\Gamma_v(1)) + \emph{rank} \: (SL_s(\mathcal{O}_v) / \Gamma_v(1))$, it remains to bound the rank of the finite quotient $SL_s(\mathcal{O}_v) / \Gamma_v(1)$. By the same argument above, this is a subgroup of $SL_{sd}(\mathbb{F}_p)$. By \cite[Cor. 24, p. 326]{ls12}, $\text{rank}(GL_{sd}(\mathbb{F}_{p})) \leq 2s^2d^2$, and so we get the bound in $(a)$ with $c(d,s) = 3d^2s^2$. The second part follows immediately as the number of finite places above $p$ is at most $d$, the degree of the field extension.
\end{proof}

Let us note that the proposition immediately implies the bound needed in \ref{(ii)}, i.e., a bound on $d(\Gamma / \Lambda(m\mathcal{O}))$. Indeed, $\Gamma / \Lambda(m\mathcal{O}) = \overline{\Gamma} / \overline{\Lambda(m \mathcal{O}})$ is a quotient of an open subgroup of $\prod_{p | m} \prod_{v | p} P_v $. By the prime number theorem, the number of primes dividing $m$ is bounded by $\frac{\log m}{\log \log m}$, and as $P_v \subset SL_s(\mathcal{O}_v)$ where $s$ is some fixed number such that $H \subset SL_s(-)$, Proposition \ref{rank fo SL} implies that $\text{rank}(\prod_{v | p} P_v) \leq d C_9$. Altogether, we have that $d(\Gamma / \Lambda(m\mathcal{O})) \leq C_{10}\frac{\log m}{\log \log m}$, as needed.\\


In order to bound $d(\Lambda(m\mathcal{O}))$ we need a more delicate argument. Let us formulate it as a Lemma:

\begin{lem}
\label{final lemma}
Let $H$ be a simple Lie group of higher rank, and $\Lambda(m \mathcal{O})$ a principal congruence subgroup of a non-uniform principal arithmetic group $\Lambda$ in $H$ as above, then there exists a constant $C_{10}= C_{10}(H)$ such that $$d({\Lambda( m\mathcal{O})}) \leq C_{10} \dfrac{\log m}{ \log \log m}$$. 
\end{lem}

\begin{proof}
Following the discussion in Section \ref{section of congruence}, $d(\Lambda(m \mathcal{O})) \leq d(\overline{\Lambda(m \mathcal{O}}))+4$ where $\overline{\Lambda(m \mathcal{O})}$
is the congruence completion of $\Lambda(m\mathcal{O})$. If $m = \prod_{i =1}^l p_i^{e_i}$, then by the strong approximation theorem, $\overline{\Lambda(m \mathcal{O}}) = A \times B \times C$ where:
\begin{itemize}
\item $A = \prod_{i=1}^l A_{p_i}$, 
where each $A_{p_i}$ is a product of pro-$p_i$ groups, one for each $v$ dividing $p_i$.
\item $B = \prod_{v \in T} P_v$ where $T = \{ P_v \; \text{is not hyper-special, and} \; v \nmid m \}$
\item $C = \prod_{v \in T'} P_v$ where $T' = \{ P_v \; \text{is hyper-special, and} \; v \nmid m \}$.
\end{itemize}

We have to bound each of $d(A), d(B)$, and $d(C)$.\\

For $A$: this is a product of (finitely many) pro-p groups, for different p's, each of rank at most $d \, C_9$ by Proposition \ref{rank fo SL}, so its rank, being a pro-nilpotent group, is also bounded by $d\,C_9$.
For bounding $B$, we recall \cite[Prop. 4.1]{bl12} and the discussion in \cite[Section 6.2]{be07}, which implies that $|T| = O_H( \log m / \log \log m)$. Using again Proposition \ref{rank fo SL}, we deduce that $d(B) \leq \text{rank}(B) \leq C_{11} \log m / \log \log m$. 
Finally, let us deal with $C$. First, a warning: $C$ is a product of infinitely many $P_v$'s and its rank is not bounded. Still, $d(C)$ is bounded. To see this, let us recall the local structure of hyper-special Parahoric subgroups. 

Let $G$ be an absolutely almost-simple simply connected algebraic group over a non-archimedean local field $K_v$, and
let $P_v \subset G(K_v)$ be a hyper-special parahoric subgroup. Then $P_v = {\mathcal{G}}(\mathcal{O}_v)$, where $\mathcal{G}$
is a reductive group scheme over the valuation ring $\mathcal{O}_v \subset K_v$ with generic fiber $G$ (cf. \cite[3.8]{Tits-local}).
It is known (cf. \cite[Exp. XXII, Proposition 2.8]{SGA3}) that the reduction $\underline{\mathcal{G}}$ is also an absolutely almost-simple simply connected algebraic group over the residue field $k_v$, which is in fact quasi-split by Lang's theorem. The group $\underline{\mathcal{G}}(k_v)$ is then quasi-simple group provided $|k_v| \geq 5$ (cf. \cite[\S 4, 11]{Stein}, \cite{Tits-simple}). On the other hand,
there is the reduction map $\mathcal{G}(\mathcal{O}_v) \to \underline{\mathcal{G}}(k_v)$, the kernel of which (the congruence subgroup
modulo the valuation ideal $\mathfrak{p}_v \subset \mathcal{O}_v$) is a pro-$p$ group for $p = \mathrm{char}\: k_v$. Thus, $P_v$ has a normal
subgroup which is a pro-$p$ group, with the quotient being a quasi-simple group.\\

In our situation, this means that if $P_v$ is hyper-special, then $P_v$ is an extension of $P_v(1)$, a pro-p group of bounded rank (actually bounded by $d^2s^2$), by a quasi-simple group of the form $G_v(\mathbb{F}_{q_v})$ where $\mathbb{F}_{q_v}$ is the residue field of $v$.
Now, $\prod_{v \in T'} P_v(1)$ is a product of infinitely many pro-p groups (at most $d$ for every $p$) each of rank at most $s^2d^2$,
 so $\text{rank}(\prod_{v \in T'} P_v(1)) \leq s^2d^3$.
 At the same time $C / \prod_{v \in T'} P_v(1) = \prod_{v \in T'} G_v(\mathbb{F}_{q_v})$ is a product of infinitely many quasi-simple finite groups, in which the multiplicity of every simple quotient is bounded by at most $d$. An elementary argument, keeping in mind that every finite quasi-simple group is generated by two elements, implies that the product is generated by at most $2d$ elements and thus $d(C) \leq d^3s^2+2d$.

\end{proof}

Recall that in Section \ref{section of congruence} we had $m \leq v^{C_7+1}$, so the proof of Theorem \ref{first reduction}, and hence also the proof of our main result, Theorem \ref{main theorem}, is now complete. \qed \\

The proof of Lemma \ref{final lemma} yields the following interesting observation: The bounds on $d(A)$ and $d(C)$ there were absolute (depending on $H$, but not on $m$). Thus, if $T = \emptyset$, i.e., if for all $v \in V_f$ $P_v$ is hyper-special, as it is for example if we take a Cehvalley group scheme, there is an absolute bound on the number of generators of principal congruence subgroups. One can even work out the bounds to deduce:

\begin{cor}
Let $G$ be a Chevalley group scheme and $\mathcal{O}$ the ring of integers in a number field $k$ with $[k : \Q] =d$.
 Then for every $I \mathrel{\ooalign{$\lneq$\cr\raise.22ex\hbox{$\lhd$}\cr}} \mathcal{O}$, 
$$d(G_{\mathcal{O}}(I)) \leq d \cdot \dim(G) +4 $$ 
where $G_{\mathcal{O}}(I) = \ker (G(\mathcal{O}) \rightarrow G(\mathcal{O}/ I))$.
\end{cor}

This is a pretty sharp estimate, as one can see that $d(G_{\mathcal{O}}(I)) \geq d \cdot \dim(G)$.\\

Let us just stress that the absolute bound is valid only for the principal congruence subgroups, but not for all congruence subgroups. In fact, a residually finite group with an absolute bound on the number of generators of its finite index subgroups must be virtually solvable (\cite{lm89}). See Section \ref{lower bound section} for more.

\section{Between Uniform and Non-Uniform Lattices}
\label{lower bound section}

The bound $d(\Gamma) = O_H(\frac{\log v}{\log \log v})$ with $v = \mu(H / \Gamma)$ which was proved in Theorem \ref{main theorem} for non-uniform lattices $\Gamma$ in 2-generic simple Lie groups $H$ is best possible. In fact, even if we take any lattice $\Lambda$ in $H$ and look only on its finite index subgroups we can not do better.  Moreover, this is true for every non virtually solvable group. More precisely:

\begin{prop}
Let $n \in \N$, $F$ a field of characteristic $p \geq 0$, and $\Gamma$ a finitely generated infinite subgroup of $GL_n(F)$ which is not virtually solvable. Then there exists a constant $c >0$ and finite index subgroups $\Gamma_i$ of $\Gamma$ with $n_i := [\Gamma : \Gamma_i] \rightarrow \infty$, such that 

\begin{itemize}
\item[(a)] If $p = 0$, $$d(\Gamma_i) \geq c \dfrac{\log n_i}{\log \log n_i} $$
\item[(b)] If $p >0$, $$d(\Gamma_i) \geq c \log n_i $$

\end{itemize}

\end{prop}

\begin{proof}
As $\Gamma$ is finitely generated and not virtually solvable, it has a specialization $\varphi: \Gamma \rightarrow GL_n(k)$ for some global field $k$ of characteristic $p$, where $\varphi(\Gamma)$ is also not virtually solvable (\cite[Thm 4.1]{ll04}). As proving the result for $\varphi(\Gamma)$ implies it for $\Gamma$, we can assume that $\Gamma \subset GL_n(k)$. Furthermore, if $G$ is the Zariski closure of $\Gamma$, it is not virtually solvable, and so we can divide by its radical in order to assure that $G$ is semi-simple. We are also allowed to replace $\Gamma$ by a group commensurable to it. Hence we can assume altogether that $G$ is simple, connected and even simply connected, defined over $k$. As $\Gamma$ is finitely generated, it is inside the $\mathcal{O}_S$-points of $G$, i.e., $\Gamma$ is a Zariski dense subgroup of an $S-$arithmetic subgroup $\Lambda$ of $G$.\\
Assume now that $p = \text{char}(F) = 0$. By the Strong Approximation Theorem for linear groups \cite[P. 391]{ls12}, $\Gamma$ is dense in the $S-$arithmetic group $\Lambda$ with respect to the congruence topology of $\Lambda$. More precisely, it implies that $\widehat{\Gamma}$ is mapped onto $\overline{\Lambda_0}$, where $\Lambda_0$ is a finite index subgroup of $\Lambda$, and $\overline{\Lambda_0}$ is its congruence completion. So, it suffices to prove the result for $\overline{\Lambda_0}$. From now on, $d(H)$ denotes the minimal number of topological generators of a group $H$. Again, we can replace $\Lambda_0$ by a principal arithmetic group $\Lambda$, defined similarly to the one defined in Section \ref{Section Principal}, and so $\Lambda = G(k) \cap \prod_{v \in V_f \setminus S} P_v$ where this time $v$ runs over the finite valuations which are not in the finite set $S$, and $\overline{\Lambda} = \prod_{v \in V_f \setminus S} P_v$ by the Strong Approximation Theorem.

Let $x$ be a large real number, $P(x)$ the set of rational primes less than $x$ and $m$ their product. By the prime number theorem $|P
(x)| \sim \frac{x}{\log x}$ and $m \sim e^x$. For all large enough $p \in P(x)$, $\overline{\Lambda} / \overline{\Lambda(p)}$ is a product of finite quasi-simple groups and hence its order is divisible by $2$. 
As $\overline{\Lambda} / \overline{\Lambda(m)} = \prod_{p \in P_x} \overline{\Lambda} / \overline{\Lambda(p)}$, 
the group $\overline{\Lambda} / \overline{\Lambda(m)}$ contains a subgroup isomorphic 
to $\mathbb{F}_2^l$ with $l \sim |P(x)|$.
 Let $\Gamma'$ be the pre-image of this subgroup in $\overline{\Lambda}$. 
Then $[\overline{\Lambda} : \Gamma'] \leq |\overline{\Lambda}/ \overline{\Lambda(m)}| \leq m^{C'}$ for some constant $C'$, while $d(\Gamma') \geq l \sim \frac{x}{\log x} \sim \frac{\log m}{\log \log m}$ and the proposition is proved for the first case. 
The reader may recognize the last argument is a quantitative use of the general "Lubotzky Alternative", see \cite[P. 400]{ls12}.\\

We turn now to the positive characteristic  case. Here our lower bound is stronger, and despite that, the proof is easier. The lower bound follows already from any local completion: 
As $\Gamma$ is Zariski dense in $G$, the closure of $\Gamma$ is open in the $v-$adic topology of $G(k_v)$, namely, it is commensurable with $G(\mathcal{O}_v)$. Now, $\mathcal{O}_v \cong \mathbb{F}_q[[t]]$ for some $q = p^e$ (where $e$ depends on the choice of $v$). If we denote $K = G(\mathcal{O}_v) = G( \mathbb{F}_q[[t]])$, and look at the congruence subgroup $K(i) = \ker (G( \mathbb{F}_q[[t]]) \rightarrow G( \mathbb{F}_q[t]/(ti)))$, then $K(i)$ is a subgroup of $K$ of index approximately $q^{i \dim G}$. At the same time, $[K(i), K(i)] \subset K(2i)$, and $K(i)^p \subset K(pi) \subset K(2i)$. Hence $d(K(i))$ is at least $i e_v$. (see \cite{lubsha} for more detailed arguments of this fact, and in particular Prop. 4.3 there which shows that this estimate is sharp). This proves (b).

\end{proof}

The proposition shows that, in particular, in every lattice $\Lambda$ in a non-compact simple Lie group $H$, there exists a sequence $\Lambda_i$ of co-volumes $v_i$ going to infinity with $d(\Lambda_i) \geq c \frac{\log v_i}{\log \log v_i}$. Theorem \ref{main theorem} shows that this is sharp even if we consider all the non-uniform lattices in $H$ (at least when $H$ is $2-$generic). On the other hand, Theorem \ref{lower bound} tells us that this bound can never be true if we take all the uniform lattices together. Let us recall its formulation again here:


\begin{thm}
Let $H$ be a connected simple Lie group of rank $\geq 1$. Then there exist $c>0$ and a sequence $\Gamma_i$ of uniform lattices in $H$ such that $\mu(H / \Gamma_i) \rightarrow \infty$ and $d(\Gamma_i) \geq c \log (\mu(H/ \Gamma_i))$
\end{thm}

This theorem is essentially proved in \cite[Thm 1(i)]{bl12} for a different goal. Let us therefore only sketch the proof.

\begin{proof}
It is shown there, based on the Golod-Shafarevich \cite{gs64} construction of infinite class field towers that there exists an infinite sequence of field extensions $k_i$ of $\Q$ of degree 
$d_i = d_{k_i} \rightarrow \infty$ and $rd_i :=  \mathcal{D}_{k_i}^{1/d_i} $ bounded, where $\mathcal{D}_{k_i}$ is the absolute value of the discriminant of $k_i$. Moreover, using a result of Prasad and Rapinchuk \cite{pr06}, these $k_i$ can be chosen in such a way that they give rise to principal arithmetic lattices $\Lambda_i$ of $H$ of co-volume at most $c_1^{d_i}$ (for a constant $c_1>0$ depending only on $H$), and such that for some fixed rational prime $p$, $\Lambda_i / \Lambda_i(p)$ is a quasi-semi-simple finite group of the form $G_i(p^{e_i})$ of order at most $p^{d_i \dim(H)}$. This finite group contains a (root) subgroup isomorphic to $\mathbb{F}_{p^{d_i}} \cong (\mathbb{F}_p)^{d_i}$ and the pre-image of it, $\Gamma_i$, satisfies therefore $d(\Gamma_i) \geq d_i$ while $\mu(H / \Gamma_i) \leq p^{d_i \dim H} \cdot c_1^{d_i} = (c_1p^{\dim H})^{d_i}$, which proves the Theorem.
\end{proof}

We end the paper by a remark on semi-simple groups $H$. Essentially, the proof of Theorem \ref{main theorem} works for irreducible non-uniform lattices in such $H$, except that for a general semi-simple group, the degree of the field of definition of non-uniform lattices can be larger than $2$, although still bounded. Thus, Gauss' Theorem which was used in Proposition \ref{Index of Principal} is not known. Therefore, for such (high rank) $H$, we can prove only the weaker statement, namely: for every irreducible non-uniform lattice $\Gamma$, $d(\Gamma) \leq O_H(\log \mu(H/\Gamma))$, while we still believe that the right bound is $O_H(\log(\mu(H/ \Gamma)) / \log log (\mu(H / \Gamma)))$. The reader is referred to \cite[Section 7]{bl19} for a discussion of the connection between such group theoretic/geometric conjectures and number theoretic open problems.

\bibliography{bibRaz}
\bibliographystyle{alpha}

\Addresses
\end{document}